\documentclass[11pt]{article}
\usepackage{lipsum}
\usepackage{amsfonts}
\usepackage{graphicx}
\usepackage{epstopdf}
\usepackage[bookmarks, colorlinks=true, plainpages = false, citecolor = blue,linkcolor=red,urlcolor = blue, filecolor = blue,pagebackref]{hyperref}

\usepackage{mathtools}
\usepackage{empheq}
\usepackage{amsfonts}
\usepackage{amsgen,amsmath,amsthm,amstext,amsbsy,amsopn,amssymb,subfigure,stmaryrd,mathabx}
\usepackage{comment}
\usepackage{enumitem}
\usepackage{booktabs}
\usepackage{blkarray}
\usepackage{algorithm}
\usepackage{algpseudocode}
\usepackage{url}
\usepackage{longtable}
\usepackage{multirow}
\usepackage{xcolor}

\ifpdf
\DeclareGraphicsExtensions{.eps,.pdf,.png,.jpg}
\else
\DeclareGraphicsExtensions{.eps}
\fi

\ifpdf
  \DeclareGraphicsExtensions{.eps,.pdf,.png,.jpg}
\else
  \DeclareGraphicsExtensions{.eps}
\fi

\newtheorem{Example}{Example}  

\newtheorem{Theorem}{Theorem}

\newtheorem{Lemma}{Lemma}

\DeclareMathAlphabet\mathbfcal{OMS}{cmsy}{b}{n}

\newcommand{\be}{\begin{equation}}
\newcommand{\ee}{\end{equation}}
\newcommand{\bea}{\begin{eqnarray}}
\newcommand{\eea}{\end{eqnarray}}
\newcommand{\beas}{\begin{eqnarray*}}
	\newcommand{\eeas}{\end{eqnarray*}}

\newcommand{\bbR}{\mathbb{R}}

\newcommand{\cT}{\mathcal{T}}

\newcommand{\cA}{\mathcal{A}}

\newcommand{\0}{{\mathbf{0}}}

\renewcommand{\u}{{\mathbf{u}}}
\renewcommand{\v}{{\mathbf{v}}}

\newcommand{\x}{{\mathbf{x}}}
\newcommand{\y}{{\mathbf{y}}}
\newcommand{\z}{{\mathbf{z}}}
\newcommand{\X}{{\mathbf{X}}}
\newcommand{\B}{{\mathbf{B}}}

\newcommand{\I}{{\mathbf{I}}}
\newcommand{\M}{{\mathbf{M}}}

\newcommand{\U}{{\mathbf{U}}}
\newcommand{\V}{{\mathbf{V}}}

\newcommand{\A}{{\mathbf{A}}}

\newcommand{\Z}{{\mathbf{Z}}}

\newcommand{\T}{{\mathbf{T}}}

\newcommand{\Proj}{P}

\newcommand{\bSigma}{\boldsymbol{\Sigma}}

\newcommand{\rank}{{\rm rank}}

\newcommand{\rmdiag}{{\rm diag}}

\newcommand{\bbO}{\mathbb{O}}

\makeatletter
\newcommand*{\rom}[1]{\expandafter\@slowromancap\romannumeral #1@}
\makeatother 
\allowdisplaybreaks

\begin{document}

\title{A Schatten-$q$ Low-rank Matrix Perturbation Analysis via Perturbation Projection Error Bound}

\author{Yuetian Luo$^1$, Rungang Han$^1$, and Anru R. Zhang$^{1, 2}$}

\date{}

\footnotetext[1]{Department of Statistics, University of Wisconsin-Madison. (\texttt{yluo86@wisc.edu}, \texttt{rhan32@stat.wisc.edu})}

\footnotetext[2]{Department of Biostatistics \& Bioinformatics, Duke University. (\texttt{anru.zhang@duke.edu})}

\maketitle

\bigskip

\begin{abstract}
This paper studies the Schatten-$q$ error of low-rank matrix estimation by singular value decomposition under perturbation. We specifically establish a perturbation bound on the low-rank matrix estimation via a perturbation projection error bound. Then, we establish lower bounds to justify the tightness of the upper bound on the low-rank matrix estimation error. We further develop a user-friendly sin$\Theta$ bound for singular subspace perturbation based on the matrix perturbation projection error bound. Finally, we demonstrate the advantage of our results over the ones in the literature by simulation.
\end{abstract}



\begin{sloppypar}
\section{Introduction}\label{sec:intro}

Let $\A$ be an $m$-by-$n$ real-valued matrix with singular value decomposition (SVD)
\begin{equation*}
	\A=[\U \,\, \U_{\perp} ] \left[ \begin{array}{c c}
			\bSigma_1 & 0 \\
			0 & \bSigma_2 
		\end{array} \right] \left[ \begin{array}{c}
			\V^\top \\
			\V_{\perp}^\top
		\end{array} \right],
\end{equation*}
where $\U \in \bbR^{m\times r}$, $\V \in \bbR^{n\times r}$, $[\U \,\, \U_{\perp}], [\V \,\, \V_{\perp}]$ are orthogonal and $\bSigma_1,\bSigma_2$ are (pseudo) diagonal matrices with decreasing singular values of $\A$. Suppose $\B = \A + \Z \in \bbR^{m \times n}$, where $\Z$ is some perturbation matrix. We similarly write down the SVD of $\B$ as
\begin{equation*}
	\begin{split}
		\B =& [\widehat{\U} \,\, \widehat{\U}_{\perp} ] \left[ \begin{array}{c c}
			\widehat{\bSigma}_1 & 0 \\
			0 & \widehat{\bSigma}_2 
		\end{array} \right] \left[ \begin{array}{c}
			\widehat{\V}^\top \\
			\widehat{\V}_{\perp}^\top
		\end{array} \right]
	\end{split}
\end{equation*} 
such that $\widehat{\U}$ and $\widehat{\V}$ share the same dimensions as $\U$ and $\V$, respectively. The relationship between the singular structures of $\A$ and $\B$ is a central topic in matrix perturbation theory. Since the seminal work by Weyl \cite{weyl1912asymptotische}, Davis-Kahan \cite{davis1970rotation}, Wedin \cite{wedin1972perturbation}, the perturbation analysis for singular values (i.e., $\bSigma_1,\bSigma_2$ versus $\widehat{\bSigma}_1, \widehat{\bSigma}_2$) and the leading singular vectors (i.e., $\U,\V$ versus $\widehat\U,\widehat\V$) have attracted enormous attentions. For example, \cite{vaccaro1994second,xu2002perturbation,liu2008first} studied perturbation expansion for singular value decomposition; \cite{li1998relative,li1998relative2,londre2000note,stewart2006perturbation} established the relative perturbation theory for eigenvectors of Hermitian matrices and singular vectors of general matrices;  \cite{demmel1990accurate,barlow1990computing,demmel1992jacobi,drmavc2008new} studied the numeric computation accuracy for singular values and vectors; more recently, \cite{yu2014useful,cai2018rate,cape2019two} developed several new perturbation results under specific structural assumptions motivated by emerging applications in statistics and data science. The readers are referred to \cite{stewart1990matrix,ipsen2000overview,bhatia2013matrix} for overviews of the historical development of matrix perturbation theory.

While most of the existing works focused on $\U$ and $\V$ or $\bSigma_1$ and $\bSigma_2$, there are fewer studies on the perturbation analysis of the true matrix $\A$ itself. In this paper, we consider the estimation of rank-$r$ matrix $\A$ (i.e., $\bSigma_2 = 0$) via rank-$r$ truncated SVD (i.e., best rank-$r$ approximation) of $\B$: $\widehat \A := \widehat\U \widehat\bSigma_1\widehat\V^\top$. Such a low-rank assumption and estimation method are widely used in many applications including matrix denoising \cite{gavish2014optimal,donoho2014minimax}, signal processing \cite{tufts1993estimation,jolliffe2002principal} and multivariate statistical analysis \cite{morrison1976multivariate}, etc. We focus on the estimation error in matrix Schatten-$q$ norm: $\|\widehat{\A} - \A\|_q$. A tight upper bound on $\|\widehat{\A} - \A\|_q$ can provide an important benchmark for both algorithmic and statistical analysis in the applications mentioned above; moreover, it can be used to study some other basic perturbation quantities, such as the pseudo-inverse perturbation $\|\widehat{\A}^\dagger - \A^\dagger\|_q$ \cite{wedin1973perturbation,stewart1977perturbation}. 

As a starting point, it is straightforward to apply the classical perturbation bounds for singular values and vectors to obtain an upper bound on $\|\widehat\A-\A\|_q$. For example, one can immediately have the following inequality via Wedin's $\sin\Theta$ Theorem (Eq. (4.4) in \cite{wedin1972perturbation}),
\begin{equation}\label{ineq: wedin matrix reconstruction}
	\| \widehat{\A} - \A \|_q \leq \|\Z\|_q \left(3 + \| \B - \widehat\A\|_q/\sigma_r(\B) \right).
\end{equation}
Another way is utilizing the optimality of SVD (Eckart-Young-Mirsky Theorem) and some basic norm inequalities to obtain:
\begin{equation}\label{ineq: matrix reconstruction via q-norm}
	\| \widehat{\A} - \A \|_q \leq \| \widehat{\A} - \B \|_q + \|\A - \B\|_q \leq 2 \| \A - \B  \|_q = 2 \| \Z  \|_q,
\end{equation} 
\begin{equation} \label{ineq: matrix reconstruction via spectral norm}
	\| \widehat{\A} - \A \|_q \leq (2r)^{1/q} \| \widehat{\A} - \A \| \overset{\eqref{ineq: matrix reconstruction via q-norm}}\leq 2^{1+1/q} r^{1/q} \|\Z\|.
\end{equation}
In contrast, we establish the following result in this paper:
\begin{Theorem}\label{th: low rank matrix reconstruction}
	Suppose $\B = \A + \Z$, where $\A$ is an unknown rank-$r$ matrix, $\B$ is the observation, and $\Z$ is the perturbation. Let $\widehat{\A} = \widehat\U \widehat\bSigma_1\widehat\V^\top$ be the best rank-$r$ approximation of $\B$. Then, 
	\begin{equation}\label{ineq:hat-A-A}
	\begin{split}
		\|\widehat{\A} - \A\|_q & \leq \left\{ \begin{array}{ll}
				(2^q + 1)^{1/q} \left\|\Z_{\max(r)}\right\|_q, & 1\leq q \leq 2;\\
				\sqrt{5} \left\|\Z_{\max(r)}\right\|_q, & 2\leq q < \infty;\\
				2\|\Z_{\max(r)}\|, & q=\infty.
			\end{array} \right.
	\end{split}
	\end{equation}
	Here $\Z_{\max(r)}$ is defined as the best rank-$r$ approximation of $\Z$.  
\end{Theorem}
The proof of Theorem \ref{th: low rank matrix reconstruction} relies on a careful characterization of $\|P_{\widehat\U_\perp}\A\|_q$ (where $P_{\widehat \U_\perp}$ is the projection onto the subspace spanned by $\widehat{\U}_\perp$) in Theorem \ref{th:SVD-projection}, which we refer as the \emph{perturbation projection error bound}. The details of Theorem 2 and the proof of Theorem \ref{th: low rank matrix reconstruction} will be presented in Section 2. 

The established bound \eqref{ineq:hat-A-A} is sharper than the classic results \eqref{ineq: wedin matrix reconstruction}, \eqref{ineq: matrix reconstruction via q-norm} and \eqref{ineq: matrix reconstruction via spectral norm} since $\|\Z_{\max(r)}\|_q\leq \min\{\|\Z\|_q, r^{1/q}\|\Z\|\}$ for any $\Z$. When $m, n \gg r$ and the first $r$ singular values of $\Z$ decay fast, which commonly happens in many large-scale matrix datasets \cite{udell2019big}, $\|\Z_{\max(r)}\|_q$ can be much smaller than $\|\Z\|_q, r^{1/q}\|\Z\|$ (see an example in Section \ref{sec:main-result}) so that the upper bound of \eqref{ineq:hat-A-A} can be much smaller than \eqref{ineq: wedin matrix reconstruction}, \eqref{ineq: matrix reconstruction via q-norm} and \eqref{ineq: matrix reconstruction via spectral norm}. 

Then, we further introduce two lower bounds to justify the tightness of the upper bound in Theorem \ref{th: low rank matrix reconstruction}. Specifically for any $\epsilon > 0$, $1\leq q \leq \infty$, we construct a triplet of matrices $(\A,\Z,\B)$ such that
\begin{equation}\label{ineq:general-lower-bound}
	\|\widehat{\A} - \A\|_q \geq ((2^q + 1)^{1/q} - \epsilon)\|\Z_{\max(r)}\|_q >0,
\end{equation}
which suggests that the constant in \eqref{ineq:hat-A-A} cannot be further improved for $q\in [1, 2]\cup\{\infty\}$. In addition, we introduce an estimation error lower bound to show that the rank-$r$ truncated SVD estimator (i.e., $\widehat \A$) is minimax rate-optimal over the class of all rank-$r$ matrices. 

As a byproduct of the theory in this paper, we derive a subspace (singular vectors) sin$\Theta$ perturbation bound (definition of Schatten-$q$ sin$\Theta$ distance is in Section \ref{sec: notation}) under the same condition as Theorem \ref{th: low rank matrix reconstruction}:
$$\max\left\{\|\sin \Theta(\widehat{\U}, \U)\|_q, \|\sin \Theta(\widehat{\V}, \V)\|_q\right\} \leq \frac{2\| \Z_{\max(r)} \|_q }{\sigma_r(\A)}.$$
This bound is ``user-friendly" as it is free of $\B$, $\widehat\U$, and $\widehat\V$, which are often perturbed and uncontrolled quantities in practice (see more discussions in Section \ref{sec:wedin}).

The rest of this paper is organized as follows. After a brief introduction on notation and preliminaries in Section \ref{sec: notation}, we present the proof of Theorem \ref{th: low rank matrix reconstruction} in Section  \ref{sec:main-result} and develop the corresponding lower bounds in Section \ref{sec:lower-bound}. The new $\sin\Theta$ perturbation analysis is done in Section \ref{sec:wedin}. We provide numerical studies to corroborate our theoretical findings in Section \ref{sec: numerical study}. Conclusion and discussions are made in Section \ref{sec: conclusion}. 

\subsection{Notation and Preliminaries} \label{sec: notation}

The following notation will be used throughout this paper. The lowercase letters (e.g., $a, b$), lowercase boldface letters (e.g., $\u, \v$), uppercase boldface letters (e.g., $\U, \V$) are used to denote scalars, vectors, matrices, respectively. For any two numbers $a,b$, let $a \wedge b = \min\{a,b\}$, $a \vee b = \max\{a,b\}$. For any matrix $\A \in \mathbb{R}^{m\times n}$ with singular value decomposition $\sum_{i=1}^{m \land n} \sigma_i(\A)\u_i \v_i^\top$, let $\A_{\max(r)}= \sum_{i=1}^{r} \sigma_i(\A)\u_i \v_i^\top$ be the best rank-$r$ approximation of $\A$, and $\A_{-\max(r)} = \sum_{i=r+1}^{m \land n} \sigma_i(\A)\u_i \v_i^\top$ be the remainder. For $q \in [1, \infty]$, the Schatten-$q$ norm of matrix $\A$ is defined as $\|\A\|_q := \left( \sum_{i=1}^{m \land n} \sigma^q_i(\A)  \right)^{1/q}$. Especially, Frobenius norm $\|\cdot\|_F$ and spectral norm $\|\cdot\|$ are Schatten-$2$ norm and Schatten-$\infty$ norm, respectively. In addition, let $\I_r$ be the $r$-by-$r$ identity matrix. Let $\bbO_{r}$ be the set of $r$-by-$r$ orthogonal matrices, $\mathbb{O}_{p, r} = \{\U \in \mathbb R^{p\times r}: \U^\top \U=\I_r\}$ be the set of all $p$-by-$r$ matrices with orthonormal columns. For any $\U\in \mathbb{O}_{p, r}$, $P_{\U} = \U\U^\top$ is the projection matrix onto the column span of $\U$. We also use $\U_\perp\in \mathbb{O}_{p, p-r}$ to represent the orthonormal complement of $\U$. We use bracket subscripts to denote sub-matrices. For example, $\A_{[i_1,i_2]}$ is the entry of $\A$ on the $i_1$-th row and $i_2$-th column; $\A_{[(r+1):m, :]}$ contains the $(r+1)$-th to the $m$-th rows of $\A$.

We use the $\sin \Theta$ norm to quantify the distance between singular subspaces. Suppose $\U_1$ and $\U_2$ are two $p$-by-$r$ matrices with orthonormal columns. Let the singular values of $\U_1^\top \U_2$ be $\sigma_1 \geq \sigma_2 \geq \ldots \geq \sigma_r \geq 0$. Then $\Theta(\U_1, \U_2)$ is defined as a diagonal matrix with principal angles between $\U_1$ and $\U_2$:
$$\Theta(\U_1, \U_2) = \rmdiag\left( \cos^{-1} (\sigma_1), \ldots, \cos^{-1} (\sigma_r) \right).$$
Then the Schatten-$q$ $\sin\Theta$ distance is defined as
\begin{equation}\label{label:Schatten-q-distance}
\|\sin\Theta(\U_1, \U_2)\|_q = \left\|\rmdiag(\sin \cos^{-1}(\sigma_1), \ldots, \sin \cos^{-1}(\sigma_r) )\right\|_q = \left(\sum_{i=1}^r (1-\sigma_r^2)^{q/2}\right)^{1/q}.
\end{equation}
Importantly, $\| \U_{1 \perp}^\top \U_2\|_q = \left\| \sin \Theta(\U_1, \U_2) \right\|_q$ for any $q \in [1,\infty]$ \cite[Lemma 2.1]{li1998relative2}. 

Finally, a function $\Phi: \bbR^n \to \bbR$ is called a symmetric gauge function if (1) $\x \neq 0 \Longrightarrow \Phi(\x) > 0 $, (2) $\Phi(\rho \x) = |\rho| \Phi(\x)$ for $\rho \in \bbR$, (3) $\Phi(\x + \y) \leq \Phi(\x) + \Phi(\y)$ for any $\x, \y \in \bbR^n$, and (4) for any permutation matrix $\mathbf{P}$, we have $\Phi(\mathbf{P}\x) = \Phi(\x)$ \cite[Definition II.3.3]{stewart1990matrix}.
\section{Proof of Theorem \ref{th: low rank matrix reconstruction}}\label{sec:main-result}
The roadmap of the proof of Theorem \ref{th: low rank matrix reconstruction} is the following. We first introduce Theorem \ref{th:SVD-projection}, which quantifies the projection error, $\|P_{\widehat{\U}_\perp}\A\|_q$ and $\|\A P_{\widehat{\V}_\perp}\|_q$, under the perturbation model. This result plays a crucial role in the proof of Theorem \ref{th: low rank matrix reconstruction} and may also be of independent interest. Next, we present Lemma \ref{lm:partition-l_q-norm} with proof and then give the proof for Theorem \ref{th: low rank matrix reconstruction}. Since the proof of Theorem \ref{th:SVD-projection} is relatively long, we present its full proof and discussions in Section \ref{sec: proof-of-th2}.  
\begin{Theorem}[A perturbation projection error bound]\label{th:SVD-projection}
	Suppose $\B = \A + \Z$ for some rank-$r$ matrix $\A$ and perturbation matrix $\Z$. Then for any $q \in [1, \infty]$,
 	\begin{equation}\label{ineq:P_U-P_V}
 	\max\left\{\|P_{\widehat{\U}_\perp}\A\|_q,\|\A P_{\widehat{\V}_\perp}\|_q \right\} \leq 2 \|\Z_{\max(r)}\|_q.
 	\end{equation}
\end{Theorem}

Next, the following Lemma \ref{lm:partition-l_q-norm} characterizes the Schatten-$q$ norm of matrix orthogonal projections. 
\begin{Lemma}\label{lm:partition-l_q-norm}
Suppose $\A,\B\in \mathbb{R}^{m\times n}$, $\U\in \mathbb{O}_{m, r}$, $q\geq 1$. Then,
$$\|P_\U\A + P_{\U_\perp} \B\|_q \leq \left\{
\begin{array}{ll}
\left(\|P_\U\A\|_q^2 + \|P_{\U_\perp}\B\|_q^2\right)^{1/2}, & 2\leq q\leq \infty;\\
\left(\|P_\U\A\|_q^q + \|P_{\U_\perp}\B\|_q^q\right)^{1/q}, & 1\leq q \leq 2.
\end{array}
\right.$$
\end{Lemma}
\begin{proof}
Let $\T = P_\U\A + P_{\U_\perp} \B$. We construct $\T_1 = P_{\U} \T = P_\U\A$, $\T_2 = P_{\U_\perp} \T = P_{\U_\perp} \B$. First we have $\T^\top \T = \T_1^\top \T_1 + \T_2^\top \T_2$. So for $p \geq 1$,
\begin{equation*}
	\|\T\|_{2p}^2 = \|\T^\top \T\|_p = \|\T_1^\top \T_1 + \T_2^\top \T_2\|_p \leq \|\T_1^\top \T_1\|_p + \|\T_2^\top \T_2\|_p = \|\T_1\|_{2p}^2 + \|\T_2\|_{2p}^2,
\end{equation*} 
which has proved the first part. 

For the second part, note that when $q = 1$, the inequality holds by triangle inequality. Next we show the inequality holds when $1 < q \leq 2$. Let $\X = \begin{bmatrix} (\T_1^\top \T_1)^{1/2} & (\T_2^\top \T_2)^{1/2}
	\end{bmatrix}$. For any $p \geq 1$ we have
\begin{equation}\label{ineq: whole big than block}
	\|\T_1^\top \T _1 + \T_2^\top \T_2\|_p^p = \| \X \X^\top \|_p^p = \|\X^\top \X\|_p^p \overset{(a)}\geq \|\T_1^\top \T_1\|_p^p + \| \T_2^\top \T_2\|_p^p, 
\end{equation} 
where (a) is because the norm of the diagonal part of a matrix is no greater than the norm of the whole matrix \cite{bhatia1988clarkson}. So we have
\begin{equation} \label{ineq: contraction of A}
   \|\T\|_{2p}^{2p} = 	\|\T^\top \T\|_{p}^{p} = \|\T_1^\top \T _1 + \T_2^\top \T_2\|_{p}^{p} \overset{\eqref{ineq: whole big than block}} \geq \|\T_1^\top \T_1\|_p^p + \| \T_2^\top \T_2\|_p^p = \|\T_1\|_{2p}^{2p} + \|\T_2\|_{2p}^{2p}.
\end{equation}
Define subspaces
 \begin{equation*}
 \begin{split}
 	\cT &= \{ \widebar{\T}: \widebar{\T} = \widebar{\T}_1 + \widebar{\T}_2, \widebar{\T}_1 = P_\U \widebar{\A}, \widebar{\T}_2 = P_{\U_\perp} \widebar{\B} \text{ and } \widebar{\A},\widebar{\B} \in \bbR^{m \times n} \},\\
 	\cT' &= \left\{  \widebar{\T}': \widebar{\T}'= \begin{bmatrix} \widebar{\T}'_1 & \0\\
\0 & \widebar{\T}'_2
\end{bmatrix}, \widebar{\T}'_1 = P_\U \widebar{\A}, \widebar{\T}'_2 = P_{\U_\perp} \widebar{\B} \text{ and } \widebar{\A},\widebar{\B} \in \bbR^{m \times n} \right\}.
\end{split}
\end{equation*}
Consider the linear map $\cA: \widebar{\T} \in \cT \longrightarrow \begin{bmatrix} P_\U \widebar{\T} & \0\\
\0 & P_{\U_\perp}\widebar{\T}
\end{bmatrix} \in \cT'.$ We can verify the adjoint map of $\cA$, $\cA^*$, satisfies $\cA^*\left(\begin{bmatrix} \widebar{\T}'_1 & \0\\
\0 & \widebar{\T}'_2
\end{bmatrix}\right) = \widebar{\T}'_1 + \widebar{\T}'_2$ as for any $\widebar{\T} \in \cT$ and $\widebar{\T}'= \begin{bmatrix} \widebar{\T}'_1 & \0\\
\0 & \widebar{\T}'_2
\end{bmatrix} \in \cT'$, we have $\langle \cA(\widebar{\T}), \widebar{\T}' \rangle = \langle P_\U \widebar{\T},  \widebar{\T}_1' \rangle +  \langle P_{\U_\perp} \widebar{\T},  \widebar{\T}_2' \rangle  = \langle \widebar{\T},\cA^*(\widebar{\T}') \rangle.$
From \eqref{ineq: contraction of A} we have shown 
$\|\cA(\T)\|_p \leq \|\T\|_p $, i.e., $\cA$ is contractive with respect to $\|\cdot\|_{p}$ for $p \geq 2$. Set $1\leq p< \infty$ such that $1/q + 1/p = 1$. Since $\|\cdot\|_p$ and $\|\cdot\|_q$ are dual norms, we have for $1<q \leq 2$:
\begin{equation*}
	\begin{split}
		\|\T\|_q = \left\| \cA^*\left(\begin{bmatrix} \T_1 & \0\\
\0 & \T_2
\end{bmatrix}\right) \right\|_q &= \sup_{\X: \|\X\|_p \leq 1} \left\langle \cA^*\left(\begin{bmatrix} \T_1 & \0\\
\0 & \T_2
\end{bmatrix}\right), \X  \right\rangle \\
&=  \sup_{\X: \|\X\|_p \leq 1} \left\langle \begin{bmatrix} \T_1 & \0\\
\0 & \T_2
\end{bmatrix},  \cA(\X)  \right\rangle\\
& \overset{(a)}\leq \sup_{\X: \|\X\|_p \leq 1} \left\| \begin{bmatrix} \T_1 & \0\\
\0 & \T_2
\end{bmatrix} \right\|_q \|\cA(\X)\|_p\\
& \leq (\|\T_1\|^q_q + \|\T_2\|_q^q)^{1/q}.  	
\end{split}
\end{equation*} Here (a) is an application of \cite[Lemma II.3.4]{stewart1990matrix} and H\"older's inequality.
This shows $\|\T\|_{q}^{q} \leq \|\T_1\|_{q}^{q} + \|\T_2\|_{q}^{q}$ and finishes the proof.
\end{proof}

Next, we prove Theorem \ref{th: low rank matrix reconstruction} based on Theorem \ref{th:SVD-projection} and Lemma \ref{th: low rank matrix reconstruction}.
\begin{proof}[Proof of Theorem \ref{th: low rank matrix reconstruction}]
For $1 \leq q < \infty$, since $\widehat{\A} = \B_{\max(r)}$ and $\widehat\U$ is composed of the first $r$ left singular vectors of $\B$, we have $\widehat\A = P_{\widehat\U}\B$ and
	\begin{equation*}
		\begin{split}
			\left\| \widehat{\A} - \A \right\|_q & = \left\|P_{\widehat{\U}}\B - P_{\widehat{\U}} \A - P_{\widehat{\U}_\perp}\A \right\|_q = \left\|P_{\widehat{\U}}\Z - P_{\widehat{\U}_\perp}\A \right\|_q\\
			& \overset{(a)} \leq \left\{ \begin{array}{ll}
				\left(\left\|P_{\widehat{\U}}\Z \right\|^q_q + \left\| P_{\widehat{\U}_\perp} \A \right\|^q_q \right)^{1/q}, & 1\leq q \leq 2;\\
				\left(\left\|P_{\widehat{\U}}\Z \right\|^2_q + \left\| P_{\widehat{\U}_\perp} \A \right\|^2_q \right)^{1/2}, & 2\leq q < \infty
			\end{array}\right. \\
			& \overset{(b)}\leq \left\{ \begin{array}{ll}
				(2^q + 1)^{1/q} \left\|\Z_{\max(r)}\right\|_q, & 1\leq q \leq 2;\\
				\sqrt{5} \left\|\Z_{\max(r)}\right\|_q, & 2\leq q < \infty.
			\end{array} \right.
		\end{split}
	\end{equation*}
Here, (a) is due to Lemma \ref{lm:partition-l_q-norm} and (b) is due to Theorem \ref{th:SVD-projection}. For $q = \infty$,
\begin{equation*}
\begin{split}
	\| \widehat{\A} - \A \| \leq \| \widehat{\A} - \B \| + \|\A - \B\| \overset{(a)}\leq 2 \| \A - \B  \| \leq 2 \| \Z  \| = 2\|\Z_{\max(r)}\|.
\end{split}
\end{equation*} 
Here $(a)$ comes from the fact that $\widehat\A$ is the best rank-$r$ approximation of $\B$.
\end{proof}

As discussed in Section \ref{sec:intro}, one can derive the matrix estimation error bounds relying on $\|\Z\|_q$ or $r^{1/q}\|\Z\|$ via the existing perturbation theory in the literature. The following example illustrates that our result can be much sharper when the singular values of $\Z$ has some polynomial decay. 
	\begin{Example}
	    Suppose $\Z$ satisfies that $\sigma_k(\Z) = k^{-1/q}$ for $q > 1$. Then
	    \begin{equation*}
	        \|\Z_{\max(r)}\|_q = \left(\sum_{k=1}^r k^{-1}\right)^{1/q} \approx (1+\log r)^{1/q},
	    \end{equation*}
        which can be much smaller than
        \begin{equation*}
	        \begin{split}
	            \|\Z\|_q & = \left(\sum_{k=1}^{m \wedge n} k^{-1}\right)^{1/q} \approx (1+\log (m \wedge n))^{1/q},\qquad    r^{1/q}\|\Z\| = r^{1/q}.
	        \end{split}
	    \end{equation*}
	\end{Example}
	
\subsection{Proof of Theorem \ref{th:SVD-projection} }\label{sec: proof-of-th2}
In the this section, we focus on the proof of Theorem \ref{th:SVD-projection}. We first introduce several additional lemmas on the properties of matrix singular values and norms, then present the proof of Theorem \ref{th:SVD-projection} and discussions. Recall that a matrix norm $\|\cdot\|$ is unitarily invariant if $\|\A\| = \|\U\A \V\|$ for any matrix $\A$ and orthogonal matrices $\U, \V$. Define $\Phi_q(\x) := \max_{1\leq i_1 <\ldots < i_r \leq n}\left(\sum_{j=1}^r |x_{i_j}|^q\right)^{1/q}$ with $q \geq 1$ for any $\x\in \bbR^n.$ We have the following Lemmas for $\|(\cdot)_{\max(r)}\|_q$. 
\begin{Lemma}\label{lem: triangle of trun schatten q} Suppose $q \geq 1$. Then $\Phi_q(\cdot)$ is a symmetric gauge function and $\|(\cdot)_{\max(r)}\|_q$ is a unitarily invariant matrix norm.
\end{Lemma}

\begin{proof} First, $\|(\cdot)_{\max(r)}\|_q$ is a unitarily invariant matrix norm follows by von Neumannan's Theory \cite[Theorem II.3.6]{stewart1990matrix} if we can show $\Phi_q(\cdot)$ is a symmetric gauge function. Recall the definition of symmetric gauge function from Section \ref{sec: notation}, it is easy to see we just need to show $\Phi_q (\x + \y) \leq \Phi_q(\x) + \Phi_q (\y)$ for any $\x,\y \in \bbR^n$. To show this, for any $\x \in \bbR^n$ and $\z \in \bbR^n$ with $\|\z\|_0 \leq r$ (here $\|\z\|_0$ denotes the number of non-zero entries in $\z$), by H\"older's inequality and the definition of $\Phi_q(\cdot)$, we have $\langle \x, \z \rangle \leq \|\z\|_p \Phi_q(\x) $ for $p$ such that $1/p + 1/q = 1$. Moreover, the equality is achieved when the equality condition in H\"older's inequality is satisfied. So
\begin{equation} \label{eq: max-r-dual-representation}
	\Phi_q(\x) = \sup_{\z: \|\z\|_p \leq 1, \|\z\|_0 \leq r} \langle \x, \z \rangle. 
\end{equation} Thus, we have
\begin{equation*}
\begin{split}
	\Phi_q (\x + \y) \overset{ \eqref{eq: max-r-dual-representation} } = \sup_{\z: \|\z\|_p \leq 1, \|\z\|_0 \leq r} \langle \x + \y, \z \rangle &\leq \sup_{\z: \|\z\|_p \leq 1, \|\z\|_0 \leq r} \langle \x, \z \rangle + \sup_{\z: \|\z\|_p \leq 1, \|\z\|_0 \leq r} \langle \y, \z \rangle\\
	& \overset{ \eqref{eq: max-r-dual-representation} } = \Phi_q(\x) + \Phi_q(\y). 
\end{split}
\end{equation*} This shows $\Phi_q(\cdot)$ is a symmetric gauge function and finishes the proof of this lemma.
\end{proof} 

Next, the following lemma introduces a dual characterization of the truncated matrix Schatten-$q$ norm. 
\begin{Lemma}[Dual representation of Truncated Schatten-$q$ norm]\label{lm: charac of Schatten-q norm}
Let $\X$ be a $m$-by-$n$ real-valued matrix. For any non-negative integer $r \leq m \wedge n$, $q \geq 1$ and $1/p + 1/q = 1$, we have
    \begin{equation}\label{eq: truncated schatten q norm}
        \|\X_{\max(r)}\|_p = \sup_{\|\B\|_q \leq 1, \rank(\B) \leq r}  \langle \B, \X \rangle.
    \end{equation} 
    If $\rank(\X) \leq r$, then 
    \begin{equation}\label{eq: schatten q norm of rank r matrix}
         \|\X\|_p = \sup_{\|\B\|_q \leq 1, \rank(\B) \leq r}  \langle \B, \X \rangle.
    \end{equation}
\end{Lemma}
\begin{proof} We first prove \eqref{eq: truncated schatten q norm}. Since $\Phi_q(\x) $ is a symmetric gauge function as we have shown in Lemma \ref{lem: triangle of trun schatten q} and its dual is $\Phi_p(\cdot)$ with $1/p + 1/q = 1$, \eqref{eq: truncated schatten q norm} follows from \cite[Lemma II.3.5]{stewart1990matrix}. Finally, \eqref{eq: schatten q norm of rank r matrix} is a special case of \eqref{eq: truncated schatten q norm}. 
\end{proof}

\begin{Lemma}\label{lm: optimality of SVD in truncated Schatten-q norm}
	Given matrix $\A \in \bbR^{m \times n}$ and any non-negative integer $k \leq m \wedge n$, for any matrix $\M$ with $\rank(\M) \leq r$, we have
	\begin{equation*}
		\left\|\left(\A_{-\max(r)}\right)_{\max(k)} \right\|_q \leq \| \left(\A - \M\right)_{\max(k)}  \|_q.
	\end{equation*} 
	The equality is achieved when $\M = \A_{\max(r)}$.
\end{Lemma}
\begin{proof}
By the well-known Eckart-Young-Mirsky Theorem \cite{eckart1936approximation,mirsky1960symmetric,golub1987generalization}, the truncated SVD achieves the best low-rank matrix approximation in any unitarily invariant norm. This lemma follows from the Eckart-Young-Mirsky Theorem and the fact that $\|(\cdot)_{\max(k)}\|_q$ is a unitarily invariant matrix norm (Lemma \ref{lem: triangle of trun schatten q}). 
\end{proof}
Now we are in position to prove Theorem \ref{th:SVD-projection}.
\begin{proof}[Proof of Theorem \ref{th:SVD-projection}]
We only study $\|P_{\widehat{\U}_\perp}\A\|_q$ since the proof of the upper bound of $\|\A P_{\widehat{\V}_\perp}\|_q$ follows by symmetry. Denote $\sum_{k=1}^r \sigma_{k}(\A)\u_k \v_k^\top$ as a singular value decomposition of $\A$. Since $\rank(P_{\widehat{\U}_\perp}\A) \leq \rank(\A)= r$, for $p \geq 1$ satisfying $1/p + 1/q = 1$, we have
\begin{align}
\left\|P_{\widehat{\U}_\perp}\A \right\|_q \overset{(a)}= & \sup_{\|\X\|_p \leq 1, \rank(\X) \leq r } \langle \Proj_{\widehat{\U}_\perp} \A, \X \rangle  \nonumber\\
= &   \sup_{\|\X\|_p \leq 1, \rank(\X) \leq r } \langle \Proj_{\widehat{\U}_\perp} \left(\A +\Z \right) - \Proj_{\widehat{\U}_\perp} \Z , \X \rangle \nonumber\\
\leq & \sup_{\|\X\|_p \leq 1, \rank(\X) \leq r } \langle \Proj_{\widehat{\U}_\perp} (\A +\Z), \X \rangle + \sup_{\|\X\|_p \leq 1, \rank(\X) \leq r } \langle \Proj_{\widehat{\U}_\perp} \Z, \X \rangle \nonumber\\
\overset{(b)}\leq & \sup_{\|\X\|_p \leq 1, \rank(\X) \leq r } \|\X\|_p \left\|  \left(\Proj_{\widehat{\U}_\perp} \left(\A +\Z \right)\right)_{\max (r)} \right\|_q \nonumber\\
 & + \sup_{\|\X\|_p \leq 1, \rank(\X) \leq r } \|\X\|_p \left\| \left(\Proj_{\widehat{\U}_\perp} \Z \right)_{\max (r)} \right\|_q\nonumber\\
\overset{(c)}\leq & \min_{\rank(\M) \leq r} \left\|  \left(\A +\Z -\M \right)_{\max (r)} \right\|_q + \left\| \left(\Proj_{\widehat{\U}_\perp} \Z \right)_{\max (r)} \right\|_q \nonumber\\
\leq & \left\| \left(\A +\Z  - P_{\U}(\A + \Z) \right)_{\max (r)}  \right\|_q + \left\| \left(\Proj_{\widehat{\U}_\perp} \Z \right)_{\max (r)} \right\|_q \nonumber\\
\leq & \left\| \left(\Proj_{\U_\perp} \Z \right)_{\max (r)} \right\|_q + \| (\Proj_{\widehat{\U}_\perp} \Z )_{\max (r)} \|_q \label{ineq:projection}\\
\leq &  \|\Z_{\max (r)}\|_q + \|\Z_{\max (r)}\|_q \leq 2\|\Z_{\max (r)}\|_q.\nonumber
\end{align}
Here (a), (b) are due to Lemma \ref{lm: charac of Schatten-q norm} and (c) is due to Lemma \ref{lm: optimality of SVD in truncated Schatten-q norm}.
\end{proof}

We make several remarks on Theorems \ref{th:SVD-projection}. 	
	
First, Theorem \ref{th:SVD-projection} may not be simply implied by the classic results. For example, the classic Wedin's $\sin\Theta$ Theorem \cite{wedin1972perturbation},
\begin{equation} \label{ineq:wedin-perturbation}
	\max\left\{\|\sin\Theta(\U, \widehat{\U})\|_q, \|\sin \Theta(\V, \widehat{\V})\|_q\right\} \leq \frac{\max \{ \|\Z\widehat{\V}\|_q, \|\widehat{\U}^\top \Z\|_q \} }{\sigma_r(\B)},
\end{equation} 
yields
\begin{equation} \label{ineq: matrix projection bound via sin Theta}
	\begin{split}
		\|P_{\widehat{\U}_\perp} \A\|_q & = \| \widehat{\U}^\top_\perp \U \bSigma_1 \V^\top \|_q \overset{(a)}\leq \| \widehat{\U}_\perp^\top \U  \|_q \sigma_1(\A) = \|\sin\Theta(\U, \widehat{\U})\|_q\sigma_1(\A) \\
		& \leq \max \left\{ \| \Z \widehat{\V}  \|_q, \|\widehat{\U}^\top \Z\|_q \right\} \frac{\sigma_1(\A) }{\sigma_r(\B)},
	\end{split}
\end{equation} here (a) is by \cite[Theorem II.3.9]{stewart1990matrix}.

This bound \eqref{ineq: matrix projection bound via sin Theta} can be less sharp or practical for its dependency on $\sigma_1(\A)/\sigma_r(\B)$. As pointed out by \cite{udell2019big}, the spectrum of large matrix datasets arising from applications often decay fast. If the singular values of $\A,\B$ decay fast, $\sigma_1(\A)/\sigma_r(\B)\gg 1$ and \eqref{ineq: matrix projection bound via sin Theta} can be loose. In contrast, our bound \eqref{ineq:P_U-P_V} in Theorem \ref{th:SVD-projection} is free of any ratio of singular values, which can be a significant advantage in practice. We will further illustrate the difference between \eqref{ineq:P_U-P_V} and \eqref{ineq: matrix projection bound via sin Theta} by simulation in Section \ref{sec: numerical matrix perturbation projection}. 

Second, it is noteworthy by \eqref{ineq:projection} in the proof of Theorem \ref{th:SVD-projection}, we have actually proved 
\begin{equation}\label{ineq:projection-error}
\begin{split}
	&\|P_{\widehat{\U}_\perp}\A\|_q \leq \left\|\left(P_{\widehat\U_{\perp}}\Z\right)_{\max(r)}\right\|_q + \left\|\left(P_{\U_{\perp}}\Z\right)_{\max(r)}\right\|_q \\ 
	& \|\A P_{\widehat{\V}_\perp}\|_q \leq   \left\|\left(\Z P_{\widehat\V_{\perp}}\right)_{\max(r)}\right\|_q + \left\|\left(\Z P_{\V_{\perp}}\right)_{\max(r)}\right\|_q
\end{split}
\end{equation}
under the setting of Theorem \ref{th:SVD-projection}. The bound \eqref{ineq:projection-error} can be better than the one in Theorem \ref{th:SVD-projection} in some scenarios. For example, when $\Z$ is (or is close to) $\U \Sigma_\Z \V^\top$ for some $r$-by-$r$ matrix $\Sigma_\Z$, the bound in \eqref{ineq:projection-error} is smaller than $\|\Z_{\max(r)}\|_q$. On the other hand, the proposed bound in Theorem \ref{th:SVD-projection} is strong enough for proving Theorem \ref{th: low rank matrix reconstruction}, does not involve $P_{\widehat\U_{\perp}}$ or $P_{\widehat\V_{\perp}}$, and can be more convenient to use.

\section{Lower Bounds}\label{sec:lower-bound}
The following Theorem \ref{th:lower bound} shows that the error upper bound for the rank-$r$ truncated SVD estimator $\widehat \A$ in Theorem \ref{th: low rank matrix reconstruction} is sharp. 
\begin{Theorem}\label{th:lower bound}
For any $\varepsilon>0$ and $q \geq 1$, there exist $\A$, $\B$, and $\Z\neq 0$ such that $\rank(\A) = r$, $\B=\A+\Z$, and
$$\|\widehat{\A} - \A\|_q > ((2^q+1)^{1/q}-\varepsilon)\|\Z_{\max(r)}\|_q.$$
\end{Theorem}
\begin{proof}
Without loss of generality we assume $0<\varepsilon<1$. We choose a value $\eta \in (0,  \frac{(2^q+1)^{1/q}}{(2^q+1)^{1/q} - \varepsilon} - 1)$. Define
\begin{equation*}
    \A = 
    \begin{bmatrix}
        2\I_r & \mathbf{0}_{r\times r} & \mathbf{0}\\
        \mathbf{0}_{r\times r} & \mathbf{0}_{r\times r} & \mathbf{0}\\
        \mathbf{0} & \mathbf{0} & \mathbf{0}
    \end{bmatrix}, \quad \Z = \begin{bmatrix}
        -(1+\eta)\I_r & \mathbf{0}_{r\times r} & \mathbf{0}\\
        \mathbf{0}_{r\times r} & \I_r & \mathbf{0}\\
        \mathbf{0} & \mathbf{0} & \mathbf{0}
    \end{bmatrix},
\end{equation*} and 
\begin{equation*}
    \B = \begin{bmatrix}
    (1-\eta)\I_r & \mathbf{0}_{r\times r} & \mathbf{0}\\
    \mathbf{0}_{r\times r} & \I_r & \mathbf{0}\\
    \mathbf{0} & \mathbf{0} & \mathbf{0}
    \end{bmatrix}.
\end{equation*}
Then, 
\begin{equation*}
    \|\widehat{\A} - \A\|_q = \left\|\begin{bmatrix}
    -2\I_r & \mathbf{0}_{r\times r}\\
    \mathbf{0}_{r\times r} & \I_r
    \end{bmatrix}\right\|_q = (2^q r + r)^{1/q},\quad \|\Z_{\max(r)}\|_q = (1+\eta) r^{1/q}. 
\end{equation*}
We thus have
\begin{equation*}
    \|\widehat{\A} - \A\|_q > ((2^q + 1)^{1/q} - \varepsilon)\|\Z_{\max(r)}\|_q.
\end{equation*}
\end{proof}
Theorems \ref{th: low rank matrix reconstruction} and \ref{th:lower bound} together imply that the constants in \eqref{ineq:hat-A-A} are not improvable when $1\leq q\leq 2$ and $q = \infty$. For $2 < p < \infty$, it would be an interesting future work to close the gap between the upper bound ($\sqrt{5}\|\Z_{\max(r)}\|_q$) and the lower bound ($(2^q+1)^{1/q}\|\Z_{\max(r)}\|_q$). 

Apart from checking the sharpness of the upper bound \eqref{ineq:hat-A-A}, another natural question is, whether the rank-$r$ truncated SVD estimator is an optimal estimator in estimating $\A$. To answer this question, we consider the minimax estimation error lower bound among all possible data-dependent procedures $\widecheck{\A} = \widecheck{\A}(\B)$ (i.e., $\widecheck{\A}$ is a deterministic or random function of matrix $\B$). We specifically focus on the following class of $(\widetilde{\A}, \widetilde{\Z}, \widetilde{\B})$ triplets:
\begin{equation*}
	\mathcal{F}_r(\xi) = \left\{ (\widetilde{\A}, \widetilde{\Z},\widetilde{\B}): \widetilde\B = \widetilde\A + \widetilde\Z,\rank(\widetilde{\A}) = r, \left\| \widetilde{\Z}_{\max(r)} \right\|_q \leq \xi \right\}.
\end{equation*} 
Here, $\xi$ corresponds to $\left\|\Z_{\max(r)}\right\|_q$ in the context of Theorem \ref{th: low rank matrix reconstruction}.
\begin{Theorem}[Schatten-$q$ minimax lower bound] \label{th: matrix reconstruction lower bound} 
For the low-rank perturbation model, if $m \wedge n \geq 2r$, then, for any $q \geq 1$, we have 
	\begin{equation*}
		\inf_{\widecheck{\A}} \sup_{(\widetilde{\A}, \widetilde{\Z},\widetilde{\B}) \in \mathcal{F}_r(\xi)} \left\| \widecheck{\A} - \widetilde{\A}  \right\|_q \geq 2^{1/q-1} \xi.
	\end{equation*}
Here the infimum is taken over all the estimation procedures. 
\end{Theorem}
\begin{proof}
The proof is done by construction. We construct 
\begin{equation*}
	\Z_1 = \left(\begin{array}{ccc}
		  \0_{r \times r} & \0 & \0\\
		\0 & \frac{\xi}{r^{1/q}} \I_r & \0\\
		\0 & \0 & \0
	\end{array} \right),\quad \widebar{\A}_1 =  \left( \begin{array}{c c c}
		  \frac{\xi }{r^{1/q}} \I_r & \0 & \0\\
		\0 & \0_{r \times r} & \0\\
		\0 & \0 & \0
	\end{array} \right),
\end{equation*} and 
\begin{equation*}
	\Z_2 =  \left( \begin{array}{c c c}
		  \frac{\xi}{ r^{1/q}}\I_r & \0 & \0\\
		\0 & \0_{r \times r} & \0\\
		\0 & \0 & \0
	\end{array} \right),\quad \widebar{\A}_2 =  \left( \begin{array}{c c c}
		  \0_{r \times r} & \0 & \0\\
		\0 & \frac{\xi }{r^{1/q}} \I_r & \0\\
		\0 & \0 & \0
	\end{array} \right).
\end{equation*}

By the construction above, we have $\| (\Z_1)_{\max(r)} \|_q = \xi$, $\|(\Z_2)_{\max(r)} \|_q = \xi$, and $\widebar{\A}_1 +  \Z_1 = \widebar{\A}_2 + \Z_2$. So
\begin{equation*}
	\begin{split}
		& \inf_{\widecheck{\A}} \sup_{(\widetilde{\A}, \widetilde{\Z}, \widetilde{\B}) \in \mathcal{F}_r(\xi)} \left\| \widecheck{\A} - \widetilde{\A}  \right\|_q \geq \inf_{\widecheck{\A}} \left(\max \left\{ \| \widecheck{\A} - \widebar{\A}_1 \|_q,  \| \widecheck{\A} - \widebar{\A}_2 \|_q  \right\} \right)\\
		\geq & \frac{1}{2} \inf_{\widecheck{\A}} \left( \| \widecheck{\A} - \widebar{\A}_1 \|_q +  \| \widecheck{\A} - \widecheck{\A}_1 \|_q  \right) \geq \frac{1}{2} \|\widebar{\A}_1 - \widebar{\A}_2\|_q = 2^{1/q-1}\xi.
	\end{split}
\end{equation*}
\end{proof} 
Combining Theorems \ref{th: low rank matrix reconstruction} and \ref{th: matrix reconstruction lower bound}, we conclude that the truncated SVD $\widehat\A$ achieves the optimal rate of low-rank matrix estimation error among all possible procedures $\widecheck{\A}$ in the class of $\mathcal{F}_r(\xi)$.

\section{Subspace Perturbation Bounds}\label{sec:wedin}
In this section, we apply the perturbation projection error bound established in Theorem \ref{th:SVD-projection} to derive a user-friendly subspace (singular vectors) perturbation bound. 
\begin{Theorem}\label{th: variant 2 of Wedin} Consider the same perturbation setting as in Theorem \ref{th: low rank matrix reconstruction}. For any $q \geq 1$, we have 
		\begin{equation*}
		\max\left\{\|\sin \Theta(\widehat{\U}, \U)\|_q, \|\sin \Theta(\widehat{\V}, \V)\|_q\right\} \leq \frac{2\| \Z_{\max(r)} \|_q }{\sigma_r(\A)}.
	\end{equation*}
\end{Theorem}
\begin{proof}
By Theorem \ref{th:SVD-projection}, we have 
\begin{equation*}
	\| P_{\widehat{\U}_\perp} \A \|_q \leq 2 \| \Z_{\max(r)} \|_q.
\end{equation*} 
Since the left singular subspace of $\A$ is $\U$, we have $\U\U^\top\A = P_{\U} \A = \A$. Then 
\begin{equation*}
	\|\sin \Theta(\widehat{\U}, \U)\|_q = \| \widehat{\U}_\perp^\top \U \|_q \overset{(a) }\leq \frac{\| \widehat{\U}_\perp^\top\U\U^\top\A \|_q}{\sigma_r(\U^\top\A)} = \frac{\|P_{\widehat\U_\perp}\A\|_q}{\sigma_r(\A)} \leq \frac{2 \| \Z_{\max(r)} \|_q}{\sigma_r(\A)},
\end{equation*} here (a) is by \cite[Theorem II.3.9]{stewart1990matrix}.
\end{proof}
We note that several similar bounds are developed towards the applications in statistics and machine learning in the past few years, for example, \cite[Corollary 4.1]{vu2013minimax}, \cite[Theorem 2]{yu2014useful}, and \cite[Lemma 5.1]{lei2015consistency}. When the matrix is positive semidefinite, these results yield
\begin{equation}\label{ineq:vu}
\begin{split}
& \left\|\sin\Theta(\widehat\U, \U)\right\|_F \leq \frac{\sqrt{2}\|\Z\|_F}{\sigma_r(\A)}, \quad \text{(\cite[Corollary 4.1]{vu2013minimax})},
\end{split}
\end{equation}
\begin{equation}\label{ineq:yu-lei}
\begin{split}
& \left\|\sin\Theta(\widehat\U, \U)\right\|_F \leq \frac{2\min\{r^{1/2}\|\Z\|, \|\Z\|_F\}}{\sigma_r(\A)} \quad \text{\cite[Theorem 2]{yu2014useful}, \cite[Lemma 5.1]{lei2015consistency}}.
\end{split}
\end{equation}
When $\A, \Z, \B$ are asymmetric, \cite{yu2014useful} also proved
\begin{equation}\label{ineq:yu}
    \begin{split}
        \left\|\sin\Theta(\widehat\U, \U)\right\|_F \leq \frac{2(2\|\A\| + \|\Z\|)\min\{r^{1/2}\|\Z\|, \|\Z\|_F\}}{\sigma_r^2(\A)} \quad \text{\cite[Theorem 3]{yu2014useful}}.
    \end{split}
\end{equation}
The perturbation bounds \eqref{ineq:vu},\eqref{ineq:yu-lei},\eqref{ineq:yu}, along with Theorem \ref{th: variant 2 of Wedin} in this paper, are ``user friendly" as they do not involve $\widehat\U$, $\widehat\V$ or $\B$ in contrast to the classical Wedin's $\sin\Theta$ bound \eqref{ineq:wedin-perturbation}. This advantage facilitates the application of these perturbations to many settings when $\A$ and $\Z$ are the given arguments: one no longer needs to further bound $\|\Z\widehat{\V}\|_q$, $\|\widehat\U^\top\Z\|_q$. The ``user friendly" advantage is also important in many settings as the denominator of \eqref{ineq:wedin-perturbation}, $\sigma_r(\B)$, depends highly on the perturbation $\Z$ and can be rather small due to perturbation \cite{yu2014useful}. In addition, our new result in Theorem \ref{th: variant 2 of Wedin} has a better dependence on both $\Z$ and $\sigma_r(\A)$ than \eqref{ineq:vu},\eqref{ineq:yu-lei},\eqref{ineq:yu} because 
$$\|\Z_{\max(r)}\|_F \leq \min\left\{r^{1/2} \|\Z\|, \|\Z\|_F\right\},$$ 
while the opposite side of this inequality does not hold. Moreover, Theorem \ref{th: variant 2 of Wedin} covers the more general asymmetric matrices in Schatten-$q$ sin$\Theta$ norms for any $q \in [1, \infty]$.

\section{Simulations} \label{sec: numerical study}

In this section, we provide numerical studies to support our theoretical results. We specifically compare the low-rank matrix estimation error bound (Theorem \ref{th: low rank matrix reconstruction}) and the matrix perturbation projection error bound (Theorem \ref{th:SVD-projection}) in Section \ref{sec:main-result} with the results in previous literature. In each setting, we randomly generate a perturbation $\Z  = \u \v^\top + \widetilde{\Z}$, draw $\A$ by a to-be-specified scheme, and construct $\B = \A + \Z$. Here $\u, \v$ are randomly generated unit vectors and $\widetilde{\Z}$ has i.i.d. $N(0, \sigma^2)$ entries. Throughout the simulation studies, we consider the Schatten-$2$ norm (i.e., Frobenius norm) as the error metric. Each simulation setting is repeated for 100 times and the average values are reported.

\subsection{Numerical Comparison of Low-Rank Matrix Estimation Error Bounds} \label{sec: numerical matrix reconstruction}
We first compare the low-rank matrix estimation error bound $\|\widehat\A - \A\|_q$ in Theorem \ref{th: low rank matrix reconstruction} and the bounds in \eqref{ineq: matrix reconstruction via q-norm} and \eqref{ineq: matrix reconstruction via spectral norm}. We set $n \in \{100,300\}, r \in \{4,6,\ldots, 16\}$, $\sigma = 0.02$, and generate $\A = \U \bSigma_1 \V^\top$, where $\U \in \bbR^{n \times r}, \V \in \bbR^{n \times r}$ are independently drawn from $\mathbb{O}_{n,r}$ uniformly at random; $\bSigma_1$ is a diagonal matrix with singular values decaying polynomially as: $(\bSigma_1)_{[i,i]} = \frac{10}{i}$, $1 \leq i \leq r$.

 The evaluations of the upper bounds in Theorem \ref{th: low rank matrix reconstruction}, \eqref{ineq: matrix reconstruction via q-norm}, \eqref{ineq: matrix reconstruction via spectral norm}, and the true value of $\|\widehat{\A} - \A\|_F$ are given in Figure \ref{fig: simulation2}. It shows that the upper bound in Theorem \ref{th: low rank matrix reconstruction} is tighter than the upper bounds in \eqref{ineq: matrix reconstruction via q-norm}, \eqref{ineq: matrix reconstruction via spectral norm} in all settings. In addition, when $n$ increases from $100$ to $300$, the upper bound of \eqref{ineq: matrix reconstruction via q-norm} significantly increases while the upper bound of Theorem \ref{th: low rank matrix reconstruction} remains steady. This is because the upper bounds of \eqref{ineq: matrix reconstruction via q-norm} and Theorem \ref{th: low rank matrix reconstruction} rely on $\|\Z\|_F$ and $\|\Z_{\max(r)}\|_F$, respectively.
\begin{figure}
	\centering
	\subfigure[$n = 100$]{\includegraphics[height = 0.3\textwidth]{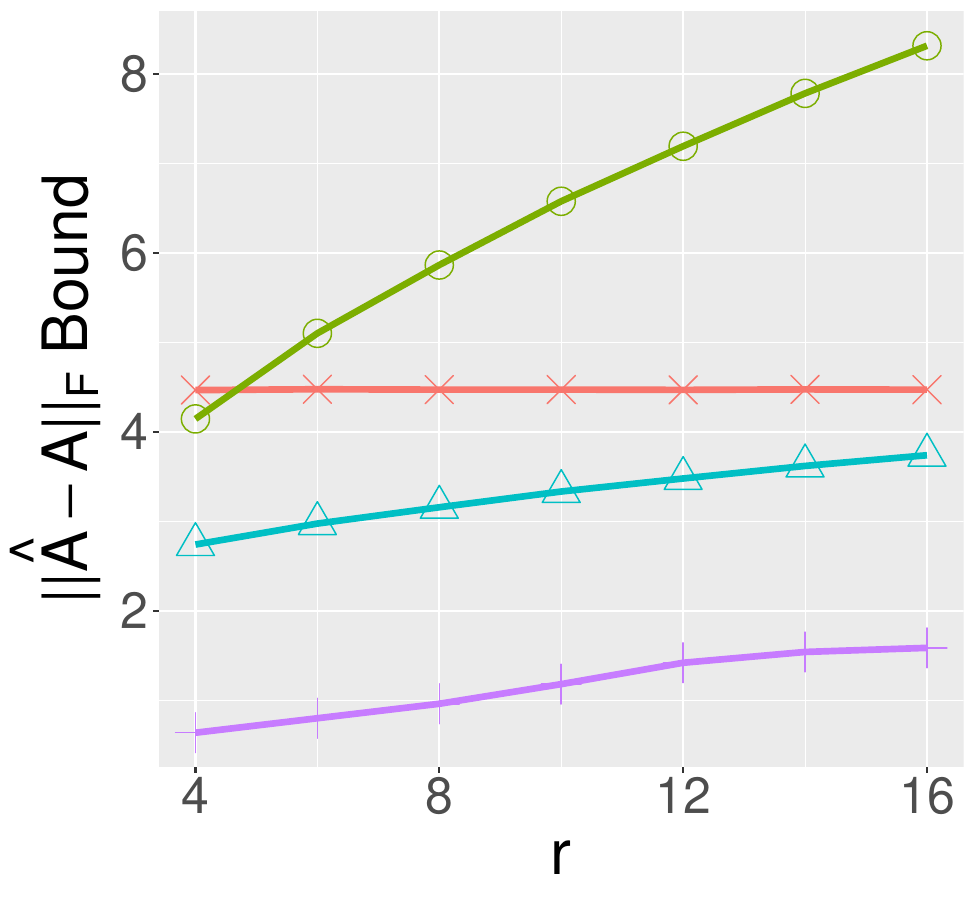}}
	\hskip1cm
	\subfigure[$n = 300$]{\includegraphics[height = 0.3\textwidth]{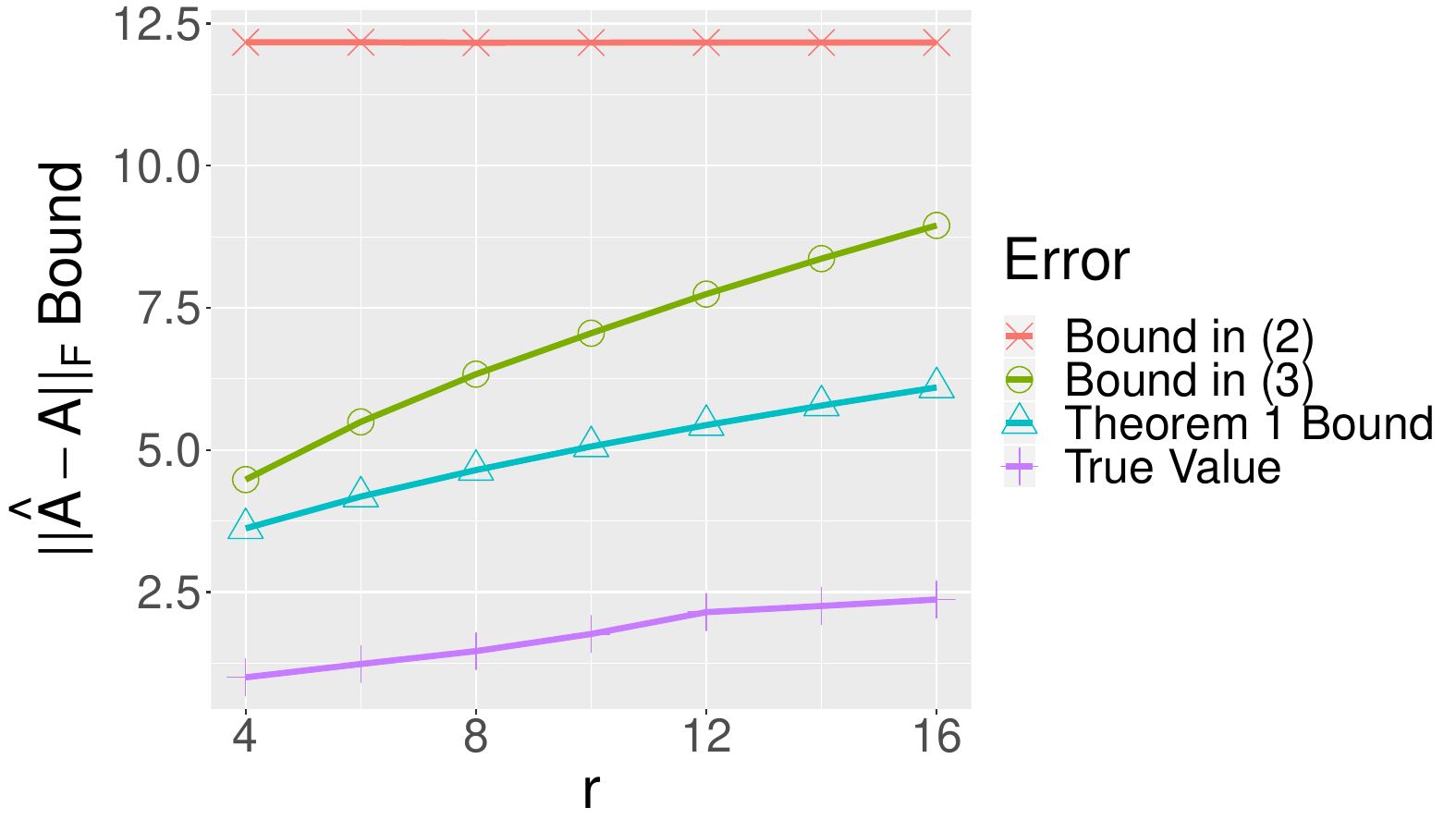}}
	\caption{Low-rank matrix estimation error bound (Theorem \ref{th: low rank matrix reconstruction}), upper bounds \eqref{ineq: matrix reconstruction via q-norm}, \eqref{ineq: matrix reconstruction via spectral norm} and the true value of $\|\widehat{\A} - \A\|_F$
	} \label{fig: simulation2}
\end{figure}

\subsection{Numerical Comparison of Matrix Perturbation Projection Error Bounds} \label{sec: numerical matrix perturbation projection}

Next, we compare the matrix perturbation projection error bound in Theorem \ref{th:SVD-projection} with the upper bound \eqref{ineq: matrix projection bound via sin Theta} derived from Wedin's sin$\Theta$ Theorem. We generate $\B, \Z$ in the same way as the previous simulation setting. When generating $\bSigma_1$ in $\A$, apart from the polynomial singular value decaying pattern considered in the last setting, we also consider the following exponential singular value decaying pattern: $(\bSigma_1)_{[i,i]} = 2^{5-i}$, $1 \leq i \leq r$. 

 The values of the upper bounds in Theorem \ref{th:SVD-projection} and \eqref{ineq: matrix projection bound via sin Theta}, along with the true value of $\|P_{\widehat{\U}_\perp}\A\|_q$, are presented in Figure \ref{fig: simulation1}. We find the bound of Theorem \ref{th:SVD-projection} is much tighter than the bound in \eqref{ineq: matrix projection bound via sin Theta}. As $r$ increases or singular value decaying pattern becomes exponential, i.e., $\A$ becomes ill-conditioned, \eqref{ineq: matrix projection bound via sin Theta} becomes loose while Theorem \ref{th:SVD-projection} can still be sharp.
\begin{figure}
	\centering
	\subfigure[Singular values of $\A$ decay polynomially]{\includegraphics[height = 0.3\textwidth]{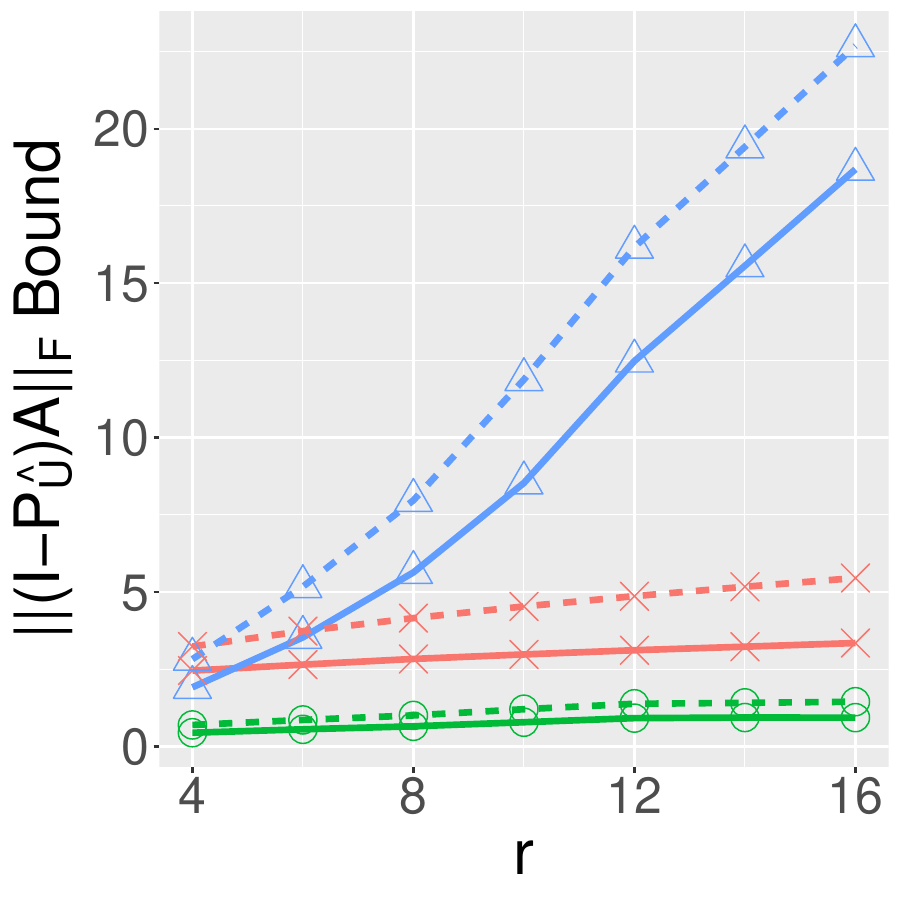}}
	\hskip1cm
	\subfigure[Singular values of $\A$ decay exponentially]{\includegraphics[height = 0.3\textwidth]{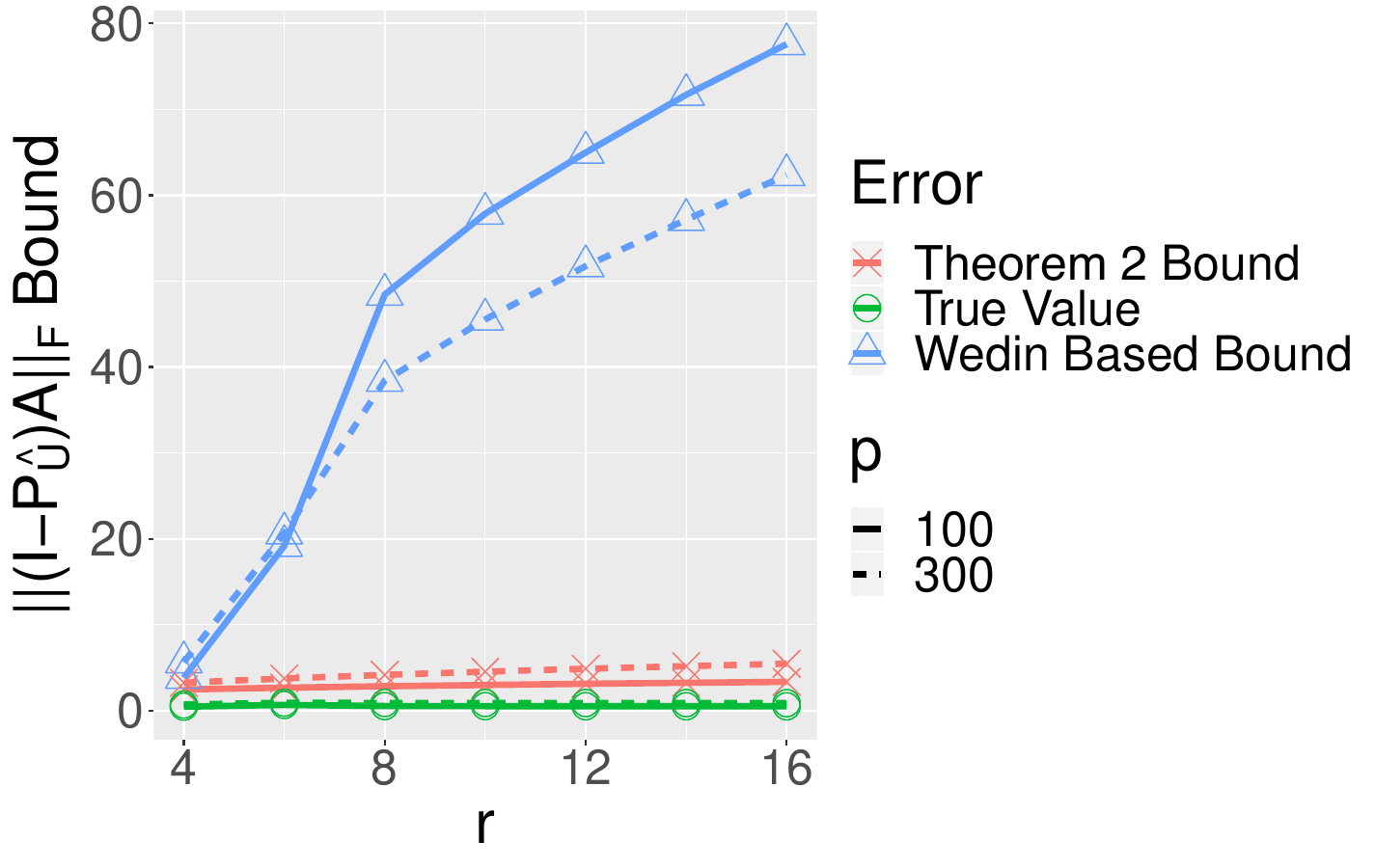}}
	\caption{Matrix perturbation projection error upper bound (Theorem \ref{th:SVD-projection}), upper bound via Wedin's sin$\Theta$ Theorem \eqref{ineq: matrix projection bound via sin Theta}, and the true value of $\|P_{\widehat\U_\perp}\A\|_F$. 
	} \label{fig: simulation1}
\end{figure}

\section{Discussions} \label{sec: conclusion}

In this paper, we prove a sharp upper bound for estimation error of rank-$r$ truncated SVD ($\|\widehat\A- \A\|_q$) under perturbation, and show its optimality in low-rank matrix estimation. The key technical tool we use is a novel matrix perturbation projection error bound for $\|P_{\widehat{\U}_\perp} \A\|_q$. As a byproduct, we also provide a sharper user-friendly sin$\Theta$ perturbation bound. The numerical studies demonstrate the advantages of these new results over the ones in the literature. 

The main result of this paper is the upper bound in \eqref{ineq:hat-A-A},  which is sharper than ones directly derived from the literature \eqref{ineq: wedin matrix reconstruction}, \eqref{ineq: matrix reconstruction via q-norm}, \eqref{ineq: matrix reconstruction via spectral norm}. We also comment that \eqref{ineq: wedin matrix reconstruction} can be conveniently extended to the general case that $\A$ is approximately rank $r$ \cite[Eq.(4.4)]{wedin1972perturbation}. It is interesting future work to study if a similar bound to \eqref{ineq:hat-A-A} can be obtained for the general approximately low-rank $\A$.

Throughout the paper, we study the additive perturbations and it is a future work to extend the results to multiplicative perturbations \cite{li1998relative,li1998relative2}. Also for convenience of presentation, we focus on the real number field in this paper. It is interesting to extend the developed results to the field of complex numbers. The main technical work for such an extension includes complex versions of Lemma \ref{lem: triangle of trun schatten q} and \ref{lm: charac of Schatten-q norm}.

Apart from the widely studied perturbation theory on singular value decomposition, the perturbation theory for other problems, such as pseudo-inverses \cite{wedin1973perturbation,stewart1977perturbation}, least squares problems \cite{stewart1977perturbation}, orthogonal projection \cite{stewart1977perturbation,xu2020perturbation,fierro1996perturbation,chen2016perturbation}, rank-one perturbation \cite{zhu2019rank}, are also important topics. It would be interesting to explore whether the tools developed in this paper is useful in studying the perturbation theory for these problems.

\bibliographystyle{alpha}
\bibliography{reference}

\end{sloppypar}

\end{document}